\documentclass[12pt]{amsart}
\usepackage{math}
\usepackage{xifthen}

\renewcommand{\E}{\ensuremath{\operatorname{E}}} 
\newcommand{\FS}{\ensuremath{\operatorname{FS}}} 
\renewcommand{\I}{\ensuremath{\mathcal{I}}} 
\renewcommand{\H}{\ensuremath{\mathcal{H}}} 
\newcommand{\K}{\ensuremath{\mathcal{K}}} 
\newcommand{\PiLeq}[1]{
	\ensuremath{\Pi_{[0, #1]}}
}
\newcommand{\Toeplitz}[2]{\ensuremath{ \PiLeq{#1} #2 \PiLeq{#1} }}
\newcommand{\ToeplitzN}[1]{\ensuremath{ \Pi_{\leq N} #1 \Pi_{\leq N} }}
\newcommand{\Inner}[1][]{\ensuremath{\langle #1\cdot,\cdot\rangle}}
\begin{document}
\nocite{*}
\title{Euclidean Embeddings and Riemannian Bergman Metrics}
\author{Eric Potash}
\address{Department of Mathematics \\Northwestern University \\ 2033 Sheridan Road Evanston, IL 60208-2730, USA}
\email{potash@math.northwestern.edu}
\begin{abstract}
	Consider the sum of the first $N$ eigenspaces for the Laplacian on a Riemannian manifold.
	A basis for this space determines a map to Euclidean space and for $N$ sufficiently large the map is an embedding.  
	In analogy with a fruitful idea of K\"ahler geometry, we define (Riemannian) Bergman metrics of degree $N$ to be those metrics induced by such embeddings.
	Our main result is to identify a natural sequence of Bergman metrics approximating any given Riemannian metric.
	In particular we have constructed finite dimensional symmetric space approximations to the space of all Riemannian metrics.
	Moreover the construction induces a Riemannian metric on that infinite dimensional manifold which we compute explicitly.
\end{abstract}
\maketitle

\section{Introduction}
\subsection{Overview}
Let $M$ be a closed smooth manifold and $\Met(M)$ the space of all smooth Riemannian metrics on $M$. The main result of this thesis is to use a reference Riemannian metric $g_0$ to define and analyze a sequence of finite dimensional approximations $\B_N$ of $\Met(M)$ which we call \textit{(Riemannian) Bergman spaces}.

Let $\H_{\leq N}$ be the sum of the first $N$ real eigenspaces for the $g_0$ Laplace-Beltrami operator, and let $d_{\leq N}$ be the dimension of this vector space.  Then an ordered basis $\Psi$ of eigenfunctions for $\H_{\leq N}$ determines a map to Euclidean space:
\[\Psi = (\psi_1,\ldots,\psi_{d_{\leq N}}): M \to \R^{d_{\leq N}}.\]
It follows from the density of eigenfunctions that, for $N$ sufficiently large, such a map is an embedding of the manifold into Euclidean space.

We define, for $N\gg 0$ the Bergman spaces $\B_N$ to be the collection of pullbacks of the Euclidean metric $g_E$ by bases of eigenfunctions:
\[\B_N = \{\Psi^*g_E \textrm{ ; }\Psi\textrm{ is a basis for }\H_{\leq N} \}.\]

If two bases differ by an orthogonal transformation then the pullback metrics are the same so we have a natural map from inner products to Bergman metrics
\[\{\textrm{inner products on }\H_{\leq N}\} \to \B_N .\]
defined by taking an inner product to the pullback of the Euclidean metric by an orthonormal basis.  Each Bergman space is thus the image of the symmetric space $\textrm{O}(\H_{\leq N})\backslash\GL(\H_{\leq N})$ in $\Met(M)$. 

The main consequence of our results is Corollary \ref{DensityCor} that the Bergman spaces are dense in $\Met(M)$ in the $C^0$ topology
\[\Met(M) \subset \lim_{N\to\infty}\overline{\B_N}\]
and there is a natural approximation map defined by (\ref{HilbDef}).  Consequently the symmetric space metrics on $\B_N$ induce a Riemannian metric on $\Met(M)$, which we shall compute explicitly in Theorem \ref{MetThm}.

Recall that $\Met(M)$ is a contractible cone in the space of all symmetric covariant two tensor fields. Hence as it is an open subset of a Fr\'echet space so it is a Fr\'echet manifold of infinite dimension \cite{Clarke}. This space, as well as its quotient by the group of diffeomorphisms of $M$, has been studied extensively. $\Met(M)$ has the structure of a fiber bundle over the space of smooth normalized volume forms on $M$ with fiber $C^\infty(M)$ \cite{Ebin}. Model theories of quantum gravity \cite{DeWitt} are concerned with the evolution of measures on $\Met(M)$. There is a natural Riemannian metric on the space of Riemannian metrics called the $L^2$ metric and for which the curvature and geodesics are known \cite{Freed}, \cite{Michor}. Geometric flows such as Ricci flow are the integral curves of a vector fields on $\Met(M)$, at least for short time. 

\subsection{Context}\label{ContextSection}
By an \textit{approximation} of a space we mean simply a sequence of finite dimensional subspaces whose union is dense, together with a compatible sequence of approximation maps.  In general approximating subspaces need not be increasing, though our Bergman spaces will have the property $\B_N \subset\overline{\B_{N+1}}$. To put our construction in context we review some important examples of approximations in analysis and geometry.

The inspiration and most relevant approximation for us is that of Bergman spaces of K\"ahler metrics. Let $(X,\omega_0)$ be a K\"ahler manifold and $\K$ the space of K\"ahler metrics cohomologous to $\omega_0$. By the $\d\bar \d$ lemma we can identify K\"ahler metrics with potentials:
\[\K = \{ \phi \in \C^\infty(X) \textrm{ ; } \omega_0 + i\d\bar\d\phi > 0\}.\]
Moreover we can identify $\K$ with the space of positive Hermitian metrics on $L$ via the map taking such a metric $h$ to its associated curvature form.

Then the theory uses an ample line bundle $L \to X$ to construct a sequence of spaces approximating $\K$ as follows.  
There is a canonical map from $X$ to the projectivization of the dual of the space of holomorphic sections of $L^k$
\begin{align*}
	\iota_k: X &\to \P(H^0(L^k)^*) \\
	x &\mapsto [s \mapsto s(x)].
\end{align*}
By the Kodaira embedding theorem, $\iota_k$ is an embedding and an inner product on $H^0(L^k)^*$ induces a K\"ahler metric on $\P(H^0(L^k)^*)$ which we may pull back.

Thus there is a map from $\I_k$, the space of Hermitian inner products on $H^0(L^k)^*$, to the space of K\"ahler metrics on $X$.  Namely in the notation of Donaldson we define
\begin{equation*}
	\FS_k: \I_k \to \K
\end{equation*}
by taking an inner product $H$ to $\frac{1}{k}\iota_k^*H$ where $H$ denotes the induced metric on the projective space.

The space of K\"ahler Bergman metrics of height $k$ is defined as the image of the $\FS_k$ map:
\begin{equation*}
	\B_k := \{\FS_k(H)\textrm{ ; }H \in \I_k \}.
\end{equation*}

There is also a map
\begin{equation*}
	\Hilb_k: \K \to \I_k
\end{equation*}
which takes a positive Hermitian metric to its associated $L^2$ inner product
\begin{equation*}
	||s||^2_{\Hilb_k(h)} = \int_{X}|s|^2_h dV(\omega_h).
\end{equation*}


Then Tian's asymptotic isometry theorem \cite{Tian} states that $\FS_k\circ \Hilb_k (h)$ tends back to $h$ in the $C^4$ topology. This was later improved to a $C^\infty$ asymptotic expansion \cite{ZelditchTian}. More precisely we have for any $r$
\begin{equation*}
	||k^{-1}\FS_k\circ \Hilb_k (h) - h||_{C^r} = O(k^{-1}).
\end{equation*}

In particular the union $\bigcup_k \B_k$ is dense in $\K$. 

Also the reference metric $\omega_0$ induces a reference Hermitian metric $h_0$ which we use to identify $\I_k$ with the symmetric space $\GL(d_k+1,\C)/\U(d_k+1)$.  Thus the construction induces a Riemannian metric on $\K$ given as a limit of pullbacks of the symmetric space metrics:
\[||\delta\omega||^2 = \lim_{k\to\infty} ||D\Hilb_k(\delta\omega)||^2. \]
In fact it follows from \cite{ZelditchTian} (cf. \cite{ChenSun}) that the limiting metric is the well-known Mabucchi-Semmes-Donaldson metric
\[||\delta\phi||^2 = k^{n+2}\int_{X} ||\delta\phi||^2 \omega_\phi^n + O(k^{n+1}).\]

Below we will briefly mention some of the many applications of these Bergman metrics to important problems in K\"ahler geometry.

A simpler example of an approximation is given by Fourier series which approximate $L^2(M)$ where $(M,g)$ is a closed Riemannian manifold. Here the approximating subspaces are the finite dimensional spaces of (smooth) functions spanned by the first $N$ eigenfunctions. The approximating maps are given by spectral truncation.

Using Fourier series we can approximate a conformal class of Riemannian metrics on a manifold.  Namely fix a reference metric $g_0 \in \Met(M)$ and let $\Conf(g_0) \subset \Met(M)$ be the conformal class of $g_0$.  Then we have the parameterization
\[\Conf(g_0) = \{e^u g_0\textrm{ ; }u \in C^\infty(M)\}.\]
and so an approximation of $C^\infty(M)$, such as Fourier series, induces an approximation of $\Conf(g_0)$.  
A similar setup where the manifold is replaced by a planar domain was used by Sheffield to define Gaussian free fields \cite{Sheffield}.

\subsection{Setup}
Let $(M,g_0)$ be a closed Riemannian manifold of dimension $n$ and $\dV_{g_0}$ the associated volume form and $\D$ the associated \textit{non-negative} Laplace-Beltrami operator. Unless noted otherwise, metric-dependent objects such as the cosphere bundle $S^*M$ are defined with respect to the reference metric $g_0$. Let $\phi_j$ be orthonormal Laplace eigenfunctions with eigenvalues $\lambda_j^2$,
\begin{align*}
	\langle \phi_j,\phi_k\rangle = \delta_{jk},\\
	\D\phi_j = \lambda_j^2 \phi_j.
\end{align*}
We will also use $\mu_N^2$ to denote the increasing sequence of \textit{distinct} eigenvalues of $\D$. Let $\H_N \subset L^2$ be the \textit{eigenfunctions of degree N}, that is the $\mu_N^2$ eigenspace
\[\H_N = \ker(\D - \mu_N^2) \]
and let $d_N = \dim \H_N$ be the multiplicty of $\mu_N^2$.  For $N > 0$ define $\H_{\leq N}$ to be the \textit{eigenfunctions up to degree N}
\[\H_{\leq N} = \bigoplus_{j=0}^N \H_{j}\]
and $d_{\leq N} = \dim \H_{\leq N}$. We endow these vector spaces with the restricted $L^2$ inner product which we denote by $\langle\cdot ,\cdot \rangle$.
Define $\Pi_N$ and $\Pi_{\leq N}$ to be orthogonal projection from $L^2(M, dV_{g_0})$ onto $\H_N$ and $\H_{\leq N}$, respectively.

Finally let 
\[\Phi_N = (\phi_{d_{\leq (N-1)+1}},\ldots,\phi_{d_{\leq N}}), \Phi_{\leq N} = (\phi_0,\ldots,\phi_{d_{\leq N}})\] 
be (ordered) orthonormal bases for $\H_N$, $\H_{\leq N}$ respectively. These may be viewed as maps from $M$ to Euclidean space with Euclidean metric $g_E$.

In Proposition \ref{EmbeddingProp} we show that $\Phi_{\leq N}$ is both injective and an immersion, that is an embedding, for $N$ sufficiently large. Thus for such $N$ we may consider the pullback of the Euclidean metric $\Phi_{\leq N}^* g_E$ on a general Riemannian manifold.

\subsection{Previous Results}
Assume that $(M,g_0)$ is an isotropy irreducible Riemannian manifold, whose definition we shall recall presently. Under this assumption Takahashi \cite{Takahashi} proved that for every $N > 0$ we have
\[\Phi_N^*g_E = C_N g_0\]
for $C_N = \mu_N^2 d_N / \Vol(M,g_0)$. That is the pullback of the Euclidean metric by the \textit{orthonormal basis} $\Phi_N$ is, up to rescaling, an isometry of $g_0$. And thus so is $\Phi_{\leq N}$:
\begin{equation}
	\Phi_{\leq N}^*g_E = \tilde C_N g_0
	\label{IsometryEq}
\end{equation}
where $\tilde C_N = \sum_{k=1}^N C_k$.

Let $\Isom (M,g_0)$ be the Riemannian isometry group. Recall that the isotropy group $\Isom_x (M,g_0)$ at a point $x$ in $M$ is the subgroup of isometries which fix $x$. This group acts (by the derivative) on the tangent space $T_x M$. A connected Riemannian manifold is said to be \textit{isotropy irreducible} if the isometry group acts transitively on $M$ and the isotropy action at each fixed point is irreducible. Hence such spaces are homogeneous and can be written as $\Isom(M,g_0) / \Isom_x (M,g_0)$. An example is the $n$ dimensional sphere $S^n$.

In the case of the round sphere, the parity of spherical harmonics implies that the metric induced by any basis for $\H_N$ is always even. Hence such metrics will not approximate all of $\Met(S^n)$. Because of this we focus on pullback metrics by bases for the larger spaces of eigenfunctions $\H_{\leq N}$.

Zelditch \cite{ZelditchWaves} gave a generalization of (\ref{IsometryEq}) by showing that the sequence of maps $\Phi_{\leq N}$ is asymptotically an isometry:
\begin{equation}
	\Phi_{\leq N}^*g_E = \frac{\Vol(S^{n-1})}{n(n+2)(2\pi)^{n}} g_0 \mu_N^{n+2} + O(\mu_N^{n+1})
	\label{PiIsometryEq}
\end{equation}
as $N\to\infty$. We will recall the proof of this result in Proposition \ref{DDPiProp} below.

A related result about embedding a manifold into $l_2$ using the heat kernel $e^{-t\D}$ was obtained in \cite{BBG}. Also of interest are so-called \textit{eigenmaps}, which are employed in a variety of applications to locally parameterize data which lie on a low dimensional manifold embedded in a high dimensional ambient Euclidean space. See \cite{JMS} for mathematical results in that direction.

\subsection{Riemannian Bergman Metrics}
Equation (\ref{PiIsometryEq}) shows that an arbitrary Riemannian metric $g_0$ may be approximated using its own eigenfunctions $\Phi_{\leq N}$.  However this approximation scheme is not finite-dimensional because the set of all $\Phi_{\leq N}$ as $g_0$ varies is not.

We remedy this by fixing $g_0$ which we call the \textit{reference metric} and studying pullbacks by arbitrary bases.  That is we define the \textit{Bergman metrics of degree} $N$ to be
\begin{equation}
	\B_N = \{ \Psi^*g_E ; \Psi\textrm{ is a basis for }\H_{\leq N}\}
	\label{}
\end{equation}
Of course if $\Psi$ is a basis for $\H_{\leq N}$ then $\Psi = Q\Phi_{\leq N}$ for some $Q \in \GL(\H_{\leq N})$. Hence if $\Phi_{\leq N}$ is an embedding, which it is for $N \gg 0$ by Proposition \ref{EmbeddingProp}, then so is $\Psi$ and so for such $N$ the Bergman metrics are Riemannian metrics.  

Also note that $\B_N \subset \overline {\B_{N+1}}$ since if $\Psi$ is a basis for $\H_{\leq N}$ it may be extended to a basis for $\H_{\leq N+1}$ as $\Psi_t = (\Psi,t\Phi_{N+1})$ and then
\begin{equation}\label{BNBNPlusOneEq}
\lim_{t\to 0}\Psi_t^*g_E \to \Psi^*g_E.
\end{equation}

An equivalent and somewhat more natural way of understanding Bergman metrics is to start with the canoncial map $\iota_{\leq N}$ from $M$ to the dual of eigenfunctions up to degree $N$
\begin{align*}
	\iota_{\leq N}: M &\to \H_{\leq N}^* \\
	x &\mapsto (\psi \mapsto \psi(x))
\end{align*}
assigning to each point the evaluation functional, and to pull back an inner product on $\H_{\leq N}^*$ to a Riemannian metric on $M$.

The $L^2$ inner product on $\H_{\leq N}$ induces an inner product on the dual $\H_{\leq N}^*$, which we also denote $\Inner$, and we have
\begin{equation}
	\iota_{\leq N}^*\Inner = \Phi_{\leq N}^*g_E.
\end{equation}
To see this let $X,Y \in T_p M$. Then we  have
\[{\iota_{\leq N}}_*(X) = \sum_{j} d\phi_j(X)\langle \cdot, \phi_j\rangle\]
so
\begin{align*}
	({\iota_{\leq N}}^*\langle \cdot, \cdot\rangle)(X,Y) &= \left\langle \sum_{j} d\phi_j(X)\langle \cdot, \phi_j\rangle\, \sum_{k}, d\phi_k(Y)\langle \cdot, \phi_k\rangle\right\rangle \\
	&= \sum_j d\phi_j(X) d\phi_j(Y) \\
	&= \Phi_{\leq N}^* g_E(X,Y)
\end{align*}
by definition of the dual inner product.  In fact that definition is all we used and so the above is true with $\Inner$ replaced by any inner product and $\Phi_{\leq N}$ with a corresponding orthonormal basis.

Concretely, if $R \in \GL(\H_{\leq N})$ is positive-definite and symmetric then $\Inner[R^{-1}]$ is another inner product on $\H_{\leq N}$, corresponding to $\Inner[R]$ on the dual where $R$ acts by precomposition. Then if $\sqrt{R}$ is a self-adjoint square root we have
\begin{equation*}
	\langle R^{-1} \sqrt{R}\phi_j, \sqrt{R}\phi_k \rangle = \delta_{jk}
\end{equation*}
so 
\begin{equation}
	\sqrt{R}\Phi_{\leq N} = (\sqrt{R}\phi_1,\ldots,\sqrt{R}\phi_{d_{\leq N}})
	\label{ONBEq}
\end{equation}
is an orthonormal basis with respect to $\Inner[R^{-1}]$ and so we have
 \begin{equation}
	\iota_{\leq N}^*\Inner[R] = (\sqrt{R}\Phi_{\leq N})^*g_E.
	\label{IotaEq}
\end{equation}

In analogy with the K\"ahler case we define a map taking inner products on $\H_{\leq N}^*$ to Riemannian Bergman metrics of degree $N$
\begin{equation}\label{ENDef}
	\E_N: \I(\H_{\leq N}^*) \to \B_N
\end{equation}
given by pulling back under $\iota_{\leq N}$:
\[\E_N(\cdot) = \iota_{\leq N}^*(\cdot).\]
Equivalently by (\ref{IotaEq}), $\E_N$ takes an inner product on $\H_{\leq N}^*$ to the metric given by pulling back the Euclidean metric on $\R^{d_{\leq N}}$ by an orthonormal basis for the corresponding inner product on $\H_{\leq N}$.

\subsection{Summary of Results}
Using the tools of semiclassical analysis we will analyze $\E_N$ applied to inner products of the form $\Inner[\Pi_N B\Pi_N]$ as $N\to\infty$.
Here $B$ is a pseudodifferential operator of order zero (hence bounded on $L^2$) so $\Pi_{\leq N} B \Pi_{\leq N}$ is a linear operator on $\H_{\leq N}$ which plays the role of $R$ in (\ref{ONBEq}), (\ref{IotaEq}) above. The inner product is implicitly restricted to $\H_{\leq N}^*$.  For convenience we will simply write $\Inner[B]$ and the inner product will be implicitly restricted to $\H_{\leq N}$.

Of course for $\Inner[B]$ to be an inner product on $\H_{\leq N}^*$ in the limit, $B$ must be symmetric and $\langle B \phi_j, \phi_j\rangle > 0$ for $j > 0$. If $B$ satisfies the last condition we say that it is \textit{positive-definite}.  Let $b(x,\xi) \in C^\infty(T^*M)$ to denote the principal symbol of $B$. 

Our main calculation is the following limit formula for the $\E_N$ map defined in (\ref{ENDef}).
\begin{theorem}{A}\label{LimitThm}
	If $B$ is a positive-definite symmetric order zero pseudodifferential operator then  
	\[ \E_N(\Inner[B])(x) = \frac{\mu_N^{n+2}}{(2\pi)^{n}(n+2)}\int_{S^*_x M} b(x,\xi) \xi\otimes\xi\ dS_{g_0}(\xi) + o(\mu_N^{n+2})\]
	as $N\to\infty$.
\end{theorem}

Here $dS_{g_0}$ is the hypersurface volume form on the $g_0$-cosphere $S^*_x M$ with respect to the $g_0(x)$ volume form on $T^*_x M$. We view the right side of the theorem as a weighted average of rank one symmetric tensors $\xi\otimes\xi$.  Each direction $\xi$ of the cosphere is weighted by the principal symbol of the pseudodifferential operator in that direction. When $B$ is the identity the average is unweighted and the theorem reduces to the asymptotic isometry (\ref{PiIsometryEq}) above.

The theorem may be equivalently reformulated in terms of orthonormal bases as follows. For $B$ as in the theorem, $\Inner[(\Pi_{\leq N} B \Pi_{\leq N})^{-1}]$ is a sequence of inner products on $\H_{\leq N}$. Then as in (\ref{ONBEq})
\[\Psi_N = \sqrt{\ToeplitzN{B}}\Phi_{\leq N}\]
is a corresponding sequence of orthonormal bases, where the square root is taken as an operator on $\H_{\leq N}$. The theorem says that as $N \to \infty$ the associated pullbacks of the Euclidean metrics are given by 
\[\Psi_N^*g_E = \frac{\mu_N^{n+2}}{(2\pi)^{n+2}(n+2)}\int_{S^*_x M} b(x,\xi) \xi\otimes\xi\ dS_{g_0}(\xi)  + o(\mu_N^{n+2}).\]

Our next result shows that the integral transform appearing on the right side of Theorem \ref{LimitThm} can be inverted.
\begin{theorem}\label{InverseThm}
	Let $g$ be a Riemannian metric on $M$. Then if $B$ is a positive-definite symmetric order zero pseudodifferential operator with principal symbol
	\begin{equation}
		b(x,\xi) = \frac{dV_{g_0}}{dV_g}||\xi||_g^{-n-2}
		\label{InverseThmSymbolEq}
	\end{equation}
	for $\xi \in S^*M$, we have
	\[\E_N( \Inner[B]) = \frac{\mu_N^{n+2}}{(n+2)(2\pi)^{n}}g + o(\mu_N^{n+2})\]
	as $N\to\infty$.
\end{theorem}
In particular the Bergman metrics are dense in the space of Riemannian metrics.  Continuing the analogy to the K\"ahler case, we will define (\ref{HilbDef}) a sequence of maps $\Hilb_N$ taking a Riemannian metric to an inner product 
\[\Hilb_N : \Met(M) \to \I(\H_{\leq N}^*)\]
and satisfying (cf. Corollary \ref{HilbCor}) 
\[\E_N \circ \Hilb_N (g) = g + o(1).\] By Theorem \ref{InverseThm} this amounts to choosing a quantization of the symbol appearing in (\ref{InverseThmSymbolEq}).

Given the $\Hilb_N$ maps we can induce a Riemannian metric on $\Met(M)$ by pulling back the symmetric space metric on $\B_N$.  Namely we define	the length of $\dot g$ in the tangent space to $g$ on $\Met(M)$ by
\begin{equation}
	||\dot g||^2 = \lim_{N\to\infty} \mu_N^{-n}||D_g \Hilb_N(\dot g)||^2.
	\label{MetMetricEq0}
\end{equation}
It is not \textit{a priori} clear that the limit exists or that the limiting metric is positive-definite. We have the following result.
\begin{theorem}\label{MetThm}
	The induced metric on $\Met(M)$ is positive-definite and given by
	\[||\dot g||^2 = \frac{1}{4n(2\pi)^{n}}\int_{S^*M} \left(\Tr_g {\dot g}+ (n+2)\frac{\langle g^{-1}\dot g g^{-1}\xi, \xi\rangle}{||\xi||^2_g}\right)^2 dS_{g_0}(x,\xi).\]
	 Moreover the metric is independent of the choice of quantization in the definition of the $\Hilb$ map.
\end{theorem}

In our context it is somewhat more natural to use the induced metric on $\Met(M)$ to measure lengths of tangent vectors which come from identifying Riemannian metrics with the metrics they induce on the cotangent bundle. Namely let $\tilde g$ be a tangent vector to $g^{-1}$ in the latter, that is $\tilde g$ is a symmetric contravariant 2-tensor. Then the identification takes $\tilde g$ to the tangent vector to $g \in \Met(M)$ given by 
	\[\left.\frac{\d}{\d\e}\right|_{\e=0} (g^{-1} + \e\tilde g)^{-1} = -g\tilde g g\]
where the right hand side is interpreted in local coordinates as multiplication of matrices and then easily seen to be invariant under change of coordinates.  In this notation Theorem \ref{MetThm} reads
	\[||\tilde g||^2 = \frac{1}{4n(2\pi)^n}\int_{S^*M} \left(\Tr{g\tilde g} + (n+2)\frac{\langle \tilde g \xi, \xi\rangle}{||\xi||^2_g}\right)^2 dS_{g_0}(x, \xi).\]

Lastly we return to bases for the individual eigenspaces $\H_N$. It is only feasible for us to analyze such maps in the case that $(M,g_0)$ is the round sphere because in that case the spectral projector $\Pi_N$ is a Fourier integral operator.

Analogously to $\iota_{\leq N}$ we define $\iota_N$ for individual eigenspaces 
\begin{align*}
	\iota_N: M &\to \H_N^* \\
	x &\mapsto (\phi \mapsto \phi(x))
\end{align*}
and prove the following theorem.

\begin{theorem}\label{SphereThm}
	Suppose $M = S^n$ is the sphere and $g_0$ is the round metric. Let $B$ be an order zero pseudodifferential operator which is positive-definite and symmetric on $\H_{N}$ for $N$ sufficiently large.  Then we have
	\[\iota_N^*\langle B\cdot, \cdot\rangle(x) = \left(\int_{S^*_xS^n} \xi\otimes\xi \int_{S^1} b\circ G^t(x,\xi)\ dt\ dS_{g_0}(\xi)\right)\left(\frac{N}{2\pi}\right)^{n+1} + O(N^{n})\]
	as $N\to\infty$ where $G^t$ is the periodic geodesic flow of $g_0$.
\end{theorem}
This generalizes Takahashi's result for orthonormal bases on the round sphere. Moreover, the theorem leads to a different proof of Theorem $\ref{LimitThm}$ with improvement in remainder to $O(\mu_N^{n+1})$ in the case of the round sphere.

\subsection{Analogy and Motivation}
Here we discuss the analogy between K\"ahler and Riemannian Bergman metrics and review some of the applications of Bergman metrics in K\"ahler geometry which we hope will help to motivate our Riemannian construction.

The basic connection between K\"ahler and Riemannian Bergman metrics is that the space of sections of high power of a line bundle $H^0(L^k)$ has been replaced by the space of eigenfunctions up to some degree $\H_{\leq N}$. Note that the definition of K\"ahler Bergman metrics depend only on the polarisation $(X,L)$. A reference metric $\omega_0$ is used only to identify the Bergman space with $\textrm{SL}(d_k+1)/\textrm{SU}(d_k+1)$. In contrast, eigenfunctions do depend on the choice of reference metric $g_0$ hence so do Riemannian Bergman metrics.

According to Tian, the philosophy of Bergman metrics in K\"ahler geometry is to use the finite dimensional geometries of projective spaces and $\GL(d_k+1)/\textrm{U}(d_k+1)$ to approximate the infinite dimensional space of K\"ahler metrics $\K$. This was done most famously by Donaldson \cite{Donaldson1}, \cite{Donaldson2} to relate the existence of K\"ahler metrics of constant scalar curvature to the existence of balanced projective embeddings.

And so naturally our hope in developing Riemannian Bergman metrics is that they will serve an analogous role. Namely that we might use the finite dimensional geometry of Euclidean space and $\textrm{O}(d_{\leq N})\backslash\GL(d_{\leq N})$ to better understand the infinite dimensional space of Riemannian metrics $\Met(M)$.  The first step in that direction is Theorem \ref{MetThm} where we compute the Riemannian metric on $\Met(M)$ induced by the Bergman spaces.



Zelditch and collaborators \cite{ZelditchRandom} proposed a way of rigorously defining and calculating path integrals over the space of K\"ahler metrics via limits of integrals over the finite dimensional K\"ahler Bergman spaces.  We speculate that our construction might be similarly used to give an approximation to Sheffield's Liouville quantum gravity \cite{Sheffield}.

\subsection{Outline}
In \S\ref{DDSection} we lay the groundwork for analyzing Bergman metrics via spectral projection kernels.  In \S\ref{IsometrySection} we recall Zelditch's asymptotic isometry. The proof of our asymptotic expansion, Theorem \ref{LimitThm}, is found in \S\ref{ResultsSection}. Its inversion, Theorem \ref{InverseThm}, is proven in \S\ref{HilbSection}. Then in \S\ref{MetricSection} we prove Theorem \ref{MetThm} calculating the induced metric on the space of metrics. Lastly in \S\ref{SphereSection} we look at individual eigenbasis maps on the round sphere and prove Theorem \ref{SphereThm}.

\section{The operator \DD}\label{DDSection}
First we prove the aforementioned fact that the eigenbasis maps are eventually embeddings. 
\begin{prop}\label{EmbeddingProp}
	For $N$ sufficiently large, if $\Psi$ is a basis for $\H_{\leq N}$ then $\Psi: M \to \R^{d_{\leq N}}$ is (i) an immersion and (ii) injective.
\end{prop}
\begin{proof}
	Because any two bases are related by a non-degenerate linear transformation, it suffices to consider $\Psi = \Phi_{\leq N}$, the orthonormal basis.

	(i) Note that if $\Phi_{\leq N}$ is not an immersion for some $N$ then neither is $\Phi_{\leq M}$ for all $M < N$.
	So supposing $\Phi_{\leq N}$ is not an immersion for arbitrarily large $N$, then for each $N$ there is a unit tangent directions $X_{N} \in S_{x_{N}}$ such that
	\[d\phi_j(x_N)(X_N) = 0 \textrm{ for all } j \leq N.\]
	Then since $SM$ is compact there is a subsequence of $\{X_N\}$ converging to $X \in T_x M$. So we have
	\[d\phi_j(x)(X) = 0 \textrm { for all } j.\]
	But this is a contradiction because $\{\phi_j\}$ is a basis for $C^\infty(M)$ in the $C^\infty$ topology (\cite{Shubin}, Proposition 20.1).
	Hence $\Phi_{\leq N}$ is an immersion for $N \gg 0$.

	(ii) Similarly, if $\Phi_{\leq N}$ is not injective for some $N$  then neither is $\phi_j$ for all $j \leq N$. So supposing $\Phi_{\leq N}$ is not injective for arbitrarily large $N$, we have for each $N$ points $x_N,y_N \in M$ such that 
	\[\phi_j(x_N) = \phi_j(y_N) \textrm{ for all } j \leq N.\]
	Then since $M$ is compact there are subsequences of $\{x_N\}$ and $\{y_N\}$ converging to $x$ and $y$, respectively. So
	\[\phi_j(x) = \phi_j(y) \textrm{ for all } j.\]
	If $x \neq y$ this contradicts the fact that $\{\phi_j\}$ is a basis.

	If $x = y$ let us work in a coordinate patch $U$ containing $x$. Then for $N$ sufficiently large $x_N, y_N \in U$. For such $N$ view $y_{N} - x_{N}$ as a tangent vector at $x_{N}$ and define 
	\[X_{N} := \frac{y_{N} - x_{N}}{||y_{N} - x_{N}||_{g_0(x_{N})}} \in S_{x_N} M.\]
	Then by compactness of $SM$ there is a subsequence $\{N_k\}$ such that $X_{N_k}\to X \in T_x X$.

	We invoke the mean value theorem: for any $k$ and every $j\leq N_k$ there is a point $z_{N_k,j}$ on the line from $x_{N_k}$ to $y_{N_k}$ such that $\phi_j$ has derivative vanishing at that point along that line: 
	\[d\phi_j(z_{N_k, j})(X_{N_k}) = 0.\]
	
	For any $j$, as $k\to\infty$ we have $(z_{N_k, j}, X_{N_k}) \to (x,X)$. So
	\[ 0 = \lim_{k\to\infty} d\phi_j(z_{N_k, j})(X_{N_k}) = d\phi_j(x)(X)\]
    for all $j$, again contradicting the fact that $\{\phi_j\}$ is a basis.

\end{proof}

Note the formula for the pullback of the Euclidean metric by a map $\Psi: M \to \R^{d_{\leq N}}$ is given by
\[\Psi^*g_E = \sum_{j=1}^{d_{\leq N}} d\Psi_j \otimes d\Psi_j.\]
It will be useful for us to have an alternate expression for this.  To that end we define a differential operator which takes a function on the product $M\times M$ and produces a covariant 2-tensor on $M$. That is let
\[\DD: C^\infty(M\times M) \to \Gamma(T^*M\otimes T^*M)\]
by
\[K(x,y) \mapsto d_x \otimes d_y K(x,y) |_{x=y}.\]

Given $X,Y \in T_p M$ we have
\[\DD K(p)(X,Y) = d_x K(p)(X) \cdot d_y K(p)(X)\]
or in local coordinates
\[\DD K(p) = \sum_{i,j=1}^n \frac{\d}{\d x_i}\frac{\d}{\d y_j}K(p,p) dx^i\otimes dx^j.\]

Note that if $K$ is a symmetric function then $\DD K$ is a symmetric tensor.
Moreover if $K$ is the kernel of an operator $Q$ on $L^2(M)$ we have
\[K(x,y) = \sum_{i,j} q_{ij}\phi_i(x)\phi_j(y)\] where
\[q_{ij} = \langle Q\phi_i, \phi_j\rangle\]
are the matrix entries of $Q$ with respect to the basis $\{\phi_j\}$. When the operator is smoothing (that is $K$ is smooth) we can apply $\DD$. In particular we have
\begin{equation}
	\DD \Pi_{\leq N} B \Pi_{\leq N} = \sum_{i,j=1}^{d_{\leq N}} b_{ij}d\phi_i \otimes d\phi_j.
	\label{DDEq}
\end{equation}
Our interest in $\DD$ stems from the following observation.
\begin{lem}
	\label{DDENLem}
	Let $Q \in \GL(\H_{\leq N})$. Then we have
	\[\DD Q^*Q = (Q\Phi_{\leq N})^*g_E\]
	where we view $Q^*Q \in \GL(\H_{\leq N})$ as an operator kernel.
	\begin{proof}
		This is a straightforward calculation.  Note that both sides of the claim are invariant under orthogonal transformation of $Q$ so let $(q_{ij})$ be a positive definite symmetric representative for $Q \in \textrm{O}(\H_{\leq N}) \backslash \GL(\H_{\leq N})$.  Then
		\begin{align*}
			(Q\Phi_{\leq N})^*g_{\E} &= \sum_{i=1}^{d_{\leq N}} d(Q\Phi_{\leq N})_i\otimes d(Q\Phi_{\leq N})_i \\
			&= \sum_{i=1}^{d_{\leq N}} \left(\sum_{j=1}^{d_{\leq N}}q_{ij}d\phi_j\right)\otimes\left(\sum_{k=1}^{d_{\leq N}}q_{ik}d\phi_k\right) \\
			&= \sum_{i,j,k=1}^{d_{\leq N}} q_{ij}q_{ik} d\phi_j\otimes d\phi_k \\
			&= \sum_{i,j,k=1}^{d_{\leq N}} q_{ji}q_{ik} d\phi_j\otimes d\phi_k \\
			&= \sum_{j,k=1}^{d_{\leq N}} (Q^*Q)_{jk} d\phi_j\otimes d\phi_k \\
			&= \DD Q^*Q
		\end{align*}
		as desired.
	\end{proof}
\end{lem}

Note that if $B$ is a positive definite and symmetric (pseudodifferential) operator on all of $L^2(M)$ we can apply the lemma to $\Pi_{\leq N}B\Pi_{\leq N}$ to obtain
\begin{equation}
	\DD \Pi_{\leq N}B\Pi_{\leq N} = (\sqrt{\Pi_{\leq N} B\Pi_{\leq N}}\Phi_{\leq N})^*g_E
	\label{}
\end{equation}

\section{Asymptotic Isometry}\label{IsometrySection}
We begin by recalling the proof of (\cite{ZelditchWaves}, Proposition 2.3) that the orthonormal basis maps are asymptotically isometric.  Note that the original result assumed aperiodicity of the geodesic flow in order to improve the remainder term.  Our proof starts with a slightly more general calculation that will be central to the Bergman asymptotics of the next section. 

Fix $\rho(\lambda)$ to be a positive Schwartz function with Fourier transform $\hat\rho(t)$ supported where $|t| < t_0$ where $t_0$ is the injectivity radius of $(M,g_0)$ and $\hat\rho(0)=1$.  Such a function exists (\cite{Hormander3}, Section 17.5) and is the staple of the Fourier Tauberian arguments which are standard in spectral asymptotics.

\begin{lem}\label{DDPiBLem}
	\[\DD[(\rho * d_\lambda \PiLeq\lambda B)(\lambda)](x) = (2\pi)^{-n}\int_{S^*_xM}b(x,\xi) \xi\otimes\xi\ \dS_{g_0}(\xi)
	\lambda^{n+1} + O(\lambda^{n})\]

	\begin{proof}
		This is a standard calculation with a parametrix for the wave group and the method of stationary phase. Let $U^t = \exp{-it\sqrt{\D}}$ and notice that
		\begin{equation}
			d_\lambda \DD \PiLeq\lambda B(x) = \sum_{j} \delta(\lambda - \lambda_j)\langle B\phi_j, \phi_k\rangle  d\phi_j(x)\otimes d\phi_k(x) 
			= (2\pi)^{-1/2}\F^{-1}_{t \to \lambda} \DD U^tB
			\label{}
		\end{equation}
		and so
		\begin{equation}
			\rho * d_\lambda \DD \PiLeq\lambda B(x) = \int_{\R} e^{it\lambda} \hat\rho(t)  \DD[U^t B](x) dt.
			\label{}
		\end{equation}

		Let $\tilde U^t$ be the Fourier integral operator parametrix given by H\"ormander \cite{Hormander4}.  That is for $|t| < t_0$ within the injectivity radius of $(M,g_0)$ we have
		\begin{equation}\label{ParametrixEq}
			\tilde U^t(x,y) = (2\pi)^{-n}\int_{T^*_x S^n}e^{i\phi(t,x,y,\eta)}a(t,x,y,\eta)d\eta
		\end{equation}
		where $d\eta$ is the Euclidean volume form with respect to the $g_0(x)$ inner product on $T^*_x M$.  For $x \in M$ let
		\[\exp_x: T_x^* M \to M\]
		denote the Riemannian exponential map and
		\[G^t: T^*M \to T^* M\] denote geodesic flow on the cotangent bundle.
		With this notation, the phase and amplitude in \ref{ParametrixEq} may be written as
		\begin{align}
			\phi(t,x,y,\eta) &= \langle \exp^{-1}_x y, \eta\rangle - t|\eta|_{g_0(x)} \\
			a(t,x,y,\eta) &= 1 + \sum_{j=1}^{\infty}\langle \eta\rangle_{g_0(x)}^{-j} a_{-j}(t,x,y)
		\end{align}
		for some smooth functions $a_j$.  We will use the following facts:
		\begin{equation}
			\begin{aligned}
				&\phi(t,x,\exp_x t\eta, \eta) = 0 \textrm{ for } |\eta|_{g_0}=1 \\
			&d_y\phi(t,x,x,\eta) = \eta \\
			&d_\eta \phi(t,x,y,\eta) = \exp_x^{-1}y - t\frac{\eta}{|\eta|_{g_0}} \\
			&d_t \phi(t,x,y,\eta) = -|\eta|_{g_0} \\
			&d_y \phi(t,x,\exp_x t\eta, \eta) = -G^t(\eta).
			\end{aligned}
		\label{DPhaseEq}
		\end{equation}
		By parametrix we mean that 
		\[\ U^t(x,y)- \tilde U^t(x,y) \in C^\infty( (-t_0, t_0) \times M\times M).\] 

		So since the Fourier transform maps Schwartz functions to Schwartz functions, we can substitute $\tilde U^t$ for $U^t$ and substitute $\eta = \lambda r \xi$ to get
		\[\rho * d_\lambda \DD \PiLeq\lambda B(x) = (2\pi)^{-n-1} \lambda^{n+1}\int_{\R}\int_{T^*_x M} \rho(t) e^{i\lambda(t+\phi(t,x,x,\xi))}  \tilde a(t,x,x,\lambda r \xi) r^{n-1} dS_{g_0}(\xi) dr dt \]
		modulo $O(\lambda^{-\infty})$. The leading term in $\lambda$ of the symbol comes from differentiating the phase and so we have
		\[\tilde a_0(0,x,x,\xi) = b(x,\xi) \xi\otimes \xi\]
		We apply stationary phase and note that the phase is the same as in the proof of the Weyl law. That is for fixed $\xi \in S^*_x M$ the unique stationary point is $r=1$ and $t=0$ and it is non-degenerate with Hessian signature zero. Hence we have
		\[\rho * d_\lambda \DD \PiLeq\lambda B(x) = (2\pi)^{-n}\lambda^{n+1}\int_{S^*_x M} b(x,\eta)\xi\otimes\xi\ dS_{g_0}(\xi) +O(\lambda^{n})\]
		as desired.
	\end{proof}
\end{lem}

We will need the following simple lemma concerning the average of rank one tensors over a sphere.

\begin{lem}\label{TensorIntegralLem}
	Let $(V,g)$ be an $n$-dimensional (real) inner product space and ${^g}S^*V \subset V^*$ the corresponding cosphere.  Then
\[\int_{{^g}{S^*V}} \xi\otimes\xi\ \dS_g(\eta) = \frac{\Vol(S^{n-1})}{n}g.\]
\begin{proof}
In normal coordinates $\{x^j\}$ where $g$ is the identity matrix we have
		\[g = \sum dx^j\otimes dx^j\]
		\[\eta\otimes\eta = \sum \xi_j\xi_k dx^j\otimes dx^k\]
		and note that for $j\neq k$
		\[\int_{{^g}S^*V}\xi_j\xi_k\ \dS_g(\xi) = 0\]
		while when $j=k$ we have 
		\[\int_{ {^g}S^*V} \xi_j\xi_j\ \dS_g(\xi) = \frac{1}{n}\int_{{^g}S^*V}|\xi|^2\ \dS_g(\xi) = \frac{\Vol(S^{n-1})}{n}.\]
\end{proof}
\end{lem}

Now we can deduce the limit of Bergman metrics corresponding to \textit{orthonormal} bases.
\begin{prop}[\cite{ZelditchWaves}, Proposition 2.3] \label{DDPiProp}
	\[E_N(\Inner) = \frac{\Vol(S^{n-1})}{n(n+2)(2\pi)^{n}}g_0(x)\mu_N^{n+2} + O(\mu_N^{n+1}) \]
	\begin{proof}
		Since
		\[d_\lambda \DD \PiLeq\lambda(x)\]
			is a positive measure we can apply the Tauberian argument of \cite{DG}, Proposition 2.1 to Lemma \ref{DDPiBLem} with $B$ the identity to get
			\[\DD \PiLeq\lambda(x) = (2\pi)^{-n}\int_{S^*_x M} \xi\otimes\xi\ \dS_{g_0}(\xi) \lambda^{n+2} + O(\lambda^{n+1})\]
		and the theorem follows from Lemma \ref{TensorIntegralLem} and the definition of $E_N$.
	\end{proof}
\end{prop}

Since $\DD \PiLeq\lambda(x)$ is increasing in $\lambda$ we can subtract to get the useful estimate
\begin{equation}
\sum_{\lambda_j \in (\lambda, \lambda+\delta]}d\phi_j(x)\otimes d\phi_j(x) = \DD \Pi_{(\lambda,\lambda+\delta]}(x) \leq C|\delta|\lambda^{n+1} + O(\lambda^{n})
	\label{DDPiDeltaEq}
\end{equation}
for some positive constant $C$.

\section{Bergman Asymptotics}\label{ResultsSection}
Next we derive the Bergman metric asymptotics:
\begin{customthm}{\ref{LimitThm}}
	If $B$ is a positive-definite and symmetric order zero pseudodifferential operator we have
	\[ \E_N( \langle B\cdot, \cdot\rangle) (x) = \left(\frac{1}{(2\pi)^{n+2}(n+2)}\int_{S^*_x M}b(x,\xi) \xi\otimes\xi\ \dS_g(\xi)
	\right)\mu_N^{n+2} + o(\mu_N^{n+2})\]
	as $N\to\infty$.
\end{customthm}

The proof of the theorem combines two standard arguments. The first, which is formalized in Lemma \ref{PiLem} below, is that 
\[\DD[\PiLeq\lambda B -\Toeplitz{\lambda}{B}] = \DD[\PiLeq\lambda B (I-\PiLeq\lambda)]\] 
is ``lower order'' and so for the sake of asymptotics it suffices to study $\DD \PiLeq\lambda B (x)$ whose smoothed out asymptotics we have already worked out.  Then we need to adapt the Tauberian argument to this setting since
\[\DD \PiLeq\lambda B \PiLeq\lambda(x)\]
is not in general non-decreasing.

Let $|\cdot|_1$ denote the trace norm and $|\cdot|_2$ denote the Hilbert-Schmidt norm for finite rank operators.
\begin{lem}\label{PiLem}
	If $B$ is an order zero pseudodifferential operator then as $\lambda\to\infty$ we have
	\begin{enumerate}[(i)]
		\item\label{PiLemi} $|\PiLeq\lambda B (I-\PiLeq\lambda)|_2 = o(\lambda^{n/2})$
		\item\label{PiLemii} $||\DD(\PiLeq\lambda B (I-\PiLeq\lambda))||_{\infty} = o(\lambda^{n+2})$.
	\end{enumerate}
	\begin{proof}
		The first claim is (\cite{Guillemin}, Lemma 3.4), which is based on the argument of Widom \cite{Widom} which we now adapt to prove the second claim.

		Note that for any $j$ and $k$ we have
		\begin{equation}
			\Pi_{\lambda_j} B \Pi_{\lambda_k} = (\lambda_j - \lambda_k)^{-1}\Pi_{\lambda_j} [\sqrt\D,B]\Pi_{\lambda_k}.
			\label{}
		\end{equation}
		We rewrite
		\begin{equation}
			\PiLeq\lambda B(I-\PiLeq\lambda) = \PiLeq\lambda B\Pi_{(\lambda+\delta,\infty)} + \PiLeq\lambda B\Pi_{[\lambda,\lambda+\delta]}.
			\label{PiLemEq1}
		\end{equation}

		To estimate the first term we let $\tilde B = [\sqrt \D, B]$ which is another order zero pseudodifferential operator and we have
		\begin{align*}
			\PiLeq\lambda B (I-\PiLeq{\lambda+\delta}) 
			&=\sum_{\substack{\lambda_j \leq \lambda \\ \lambda_k > \lambda+\delta}} (\lambda_j-\lambda_k)^{-1}\Pi_{\lambda_j}\tilde B\Pi_{\lambda_k}.
		\end{align*}
		But $(\lambda_j - \lambda_k)^{-1} \leq \delta^{-1}$ so we can bound $\DD$ of the left hand side at a point $x$ by
		\begin{equation}
			\delta^{-1} \sum_{\lambda_j \leq \lambda} |d\phi_j(x)| 
				\sum_{\lambda_k > \lambda+\delta} \left|\langle \tilde B \phi_j, \phi_k \rangle\right| |d\phi_k(x)|.
			\label{PiLemEq3}
		\end{equation}
		Repeatedly applying the commutator trick above we have
		\begin{equation}
			\langle \tilde B \phi_j, \phi_k \rangle \leq C_N |\lambda_j-\lambda_k|^{-N}
			\label{}
		\end{equation}
		where $C_N$ is the norm of $N$-fold commutator of $B$ with $\sqrt{\D}$. Also by applying the Cauchy-Schwarz inequality to (\ref{DDPiDeltaEq}) we have
		\begin{equation}
			\sum_{\lambda < \lambda_k \leq \lambda+1} |d\phi_j(x)| = O(\lambda^{\frac{n+1}{2}}).
			\label{}
		\end{equation}
		So we can break up the inner sum in equation (\ref{PiLemEq3}) into unit length segments and get the bound
		\begin{align*}
			C_N\sum_{\lambda_k > \lambda + \delta} |\lambda_k-\lambda|^{-N} |d\phi_k(x)| 
			\leq C_N \sum_{T=0}^\infty (T+\delta)^{-N} (T+\lambda+\delta+1)^{\frac{n+1}{2}}
			\label{}
		\end{align*}
		and so by picking $N = \frac{n+5}{2}$, say, the inner sum converges and is bounded by a constant times $\delta^{-2}\lambda^{\frac{n+1}{2}}$. Finally we break up the outer sum into unit length intervals
		\begin{equation}
			\delta^{-3}\sum_{\substack{T=0 \\ T < \lambda}} \sum_{\lambda_j \in [T,T+1]} |d\phi_j(x)| \lambda^{\frac{n+1}{2}} \leq |\delta|^{-3}\lambda^{n+2}.
			\label{PiLemTerm1}
		\end{equation}
		where we used Cauchy-Schwarz on the inner sum and the gradient estimate (\ref{DDPiDeltaEq}).

		To estimate the second term in (\ref{PiLemEq1}) we use the same gradient estimate and Cauchy-Schwarz:
		\begin{align}\label{PiLemTerm2}
		|\DD\Pi_{[0,\lambda]}B \Pi_{(\lambda,\lambda+\delta]}(x)| &\leq ||B||\sum_{\lambda_j \leq \lambda}\sum_{\lambda < \lambda_k \leq \lambda+\delta} |d\phi_j(x)| |d\phi_k(x)| \notag \\
		&\leq ||B|| \left(\sum_{\lambda_j \leq \lambda} |d\phi_j(x)|^2\right)^{\frac{1}{2}} \left(\sum_{\lambda < \lambda_k \leq \lambda+\delta} |d\phi_j(x)|^2\right)^{\frac{1}{2}} \\
		&\leq C|\delta|\lambda^{n+\frac{3}{2}}. \notag
		\end{align}

		Thus combining the estimates (\ref{PiLemTerm1}), (\ref{PiLemTerm2}) we have with another constant $C$
		\begin{equation}
			\lambda^{-n-2}|\DD \PiLeq\lambda B (I - \PiLeq\lambda)(x)| \leq C (|\delta|^{-3} + |\delta|\lambda^{-1/2}).
			\label{}
		\end{equation}
		and so letting $\lambda \to \infty$ we have
		\begin{equation}
			\limsup \lambda^{-n-2}|\DD \PiLeq\lambda B (I - \PiLeq\lambda)(x)| \leq \frac{C}{\delta^{3}}
			\label{}
		\end{equation}
		and by taking $\delta$ arbitrarily large we are done.
	\end{proof}
\end{lem}

\begin{proof}[Proof of Theorem \ref{LimitThm}]
	By the previous lemma it only remains to derive the asymptotics
	\begin{equation}\label{DPiBEq}
		\DD\PiLeq\lambda B(x) = \frac{1}{(2\pi)^{n+2}(n+2)}\int_{S^*_x}b(x,\xi)\ \xi\otimes\xi \dS_g.
	\end{equation}
	For a tangent vector $X \in T_x M$, 
	\[\DD\PiLeq\lambda B(x)(X, X)\]
	is not in general non-decreasing in $\lambda$, so we cannot apply the Tauberian method directly. However it is easy to see that
	\[|\DD\PiLeq\lambda B(x)| \leq |B| \DD\PiLeq\lambda(x).\]
	Hence by replacing $B$ with $B + CI$ for some $C > |B|$ we can ensure that 
	\[\DD\PiLeq\lambda (B+CI)(x)(X, X)\]
	is non-decreasing for all points $x \in M$ and tangent vectors $X \in T_x M$. Thus the Tauberian method applied to Lemma \ref{DDPiBLem} gives

	\[\DD\PiLeq\lambda (B + CI)(x) = (2\pi)^{-n}\int_{S^*_x}(b(x,\xi) + C)\ \xi\otimes\xi \dS_g.\]
	Then (\ref{DPiBEq}) follows by subtracting the asymptotics for $\DD\PiLeq\lambda$ which we computed in Proposition \ref{DDPiProp}.
\end{proof}

Note that when $B$ is multiplication by the smooth positive function $e^u$ the theorem shows that
\[E_N (\Inner[e^u])(x) = \frac{n(n+2)(2\pi)^{n+2}}{\Vol(S^{n-1})} e^{u(x)} g_0(x) + o(\mu_N^{n+2}).\]
In other words, a conformal metric $e^u g_0$ is approximated by $\E_N \Inner[e^u]$.  
Of course we would also like to approximate arbitrary, non-conformal, metrics. 

\section{The Hilb Map}\label{HilbSection}
\begin{customthm}{\ref{InverseThm}}
	Given a Riemannian metric $g$ and a positive definite symmetric order zero pseudodifferential operator $B$ with principal symbol
	\begin{equation}\label{InverseSymbolEq}
		b(x,\xi) = \frac{n(n+2)(2\pi)^{n+2}}{\Vol(S^{n-1})}\frac{dV_{g_0}}{dV_g}||\xi||_g^{-n-2}.
	\end{equation}
	for $\xi \in S^*_x M$ we have
	\[\E_N(\Inner[B] ) = \mu_N^{n+2}g + o(\mu_N^{n+2}).\]
	\begin{proof}
		We know that for any $g$ we have
		\begin{equation}
			g(x) = \frac{n}{\Vol(S^{n-1})}\int_{ {^g}S^*_x M}\xi\otimes\xi\ \dS_g(\xi)
			\label{InverseThmEq1}
		\end{equation}
		where the integral is over the cosphere ${^g}S^*_x M$ with respect to $g$ and the volume form $\dS_g$ is induced by $g$.  In light of Theorem \ref{LimitThm} we wish to rewrite this as a similar integral over the $g_0$ cosphere.  We achieve this simply by pulling the integral back under the diffeomorphism from the $g_0$-cosphere to the $g$-cosphere given by		
		\begin{equation}
			\xi \mapsto \xi/||\xi||_{g(x)}.
			\label{InverseThmEq0}
		\end{equation}

		Let us work in Riemannian normal coordinates for $g_0$ at $x$ so that $g_0(x)$ is the identity and let $g(x) = Q^*Q$ in these coordinates.  Then $Q$ is an isometry from the $g$ cosphere to the $g_0$ cosphere and we proceed by factoring (\ref{InverseThmEq0}) as
		\begin{equation}\xymatrix{
			S^*_x M \ar[r]_{\frac{Q\cdot}{||Q^\cdot||}} &S^*_x M \ar[r]_{Q^{-1}\cdot}^{\sim} & {^g}S^*_x M 
		}\end{equation}
		So we first pull back under $Q^{-1}$, substituting $\xi = Q^{-1}\eta$.
		\begin{align*}
			\int_{ {^g}S^*_x M}\xi\otimes\xi\ \dS_g(\eta) &= \int_{S^*_x M} Q^{-1}\eta \otimes Q^{-1}\eta\ \dS_{g_0}(\eta)
		\end{align*}
		Next we pull this integral back under the diffeomorphism $f$ of $S^*_x M$ given by
		\[f(\eta) = Q\eta/||Q\eta||\]
		to see that
		\begin{equation}
			\int_{ {^g}S^*_x M}\xi\otimes\xi\ \dS_g(\xi) = \int_{S^*_x M} ||Q\eta||^{-2} \eta\otimes\eta\ f^*\dS_{g_0}(\eta).
			\label{InverseThmEq2}
		\end{equation}
		
		It remains to compute $f^*\dS_{g_0}$.  For this we first observe that
		\begin{equation*}
			f(\eta + \e\zeta) = \frac{Q\eta}{||Q\eta||} +\left(\frac{Q\zeta}{||Q\eta||} - \frac{\langle Q\eta, Q\zeta\rangle}{||Q\eta||^3} Q\eta\right)\e+ O(\e^2).
		\end{equation*}
		Next recall that $\dS_{g_0}(\eta) = \eta \lrcorner \dV_{g_0}$ where we view $\eta$ as the normal vector to $\eta$ in $S^*_x M$, so if $\zeta_1,\ldots,\zeta_{n-1}$ are tangent vectors to $\eta$ in $S^*_x M$ we have
		\begin{equation*}
			\begin{aligned}
			f^*\dS_{g_0}(f_*\zeta_1,\ldots,f_*\zeta_n) &= \dV_{g_0}(f(\eta),f_*\zeta_1,\ldots,f_*\zeta_{n-1}) \\
			&= \dV_{g_0}\Bigg(\frac{Q\eta}{||Q\eta||}, \frac{Q\zeta_1}{||Q\eta||} + \frac{\langle Q\eta, Q\zeta_1\rangle}{||Q\eta||^3} Q\eta, \\
			&\hspace{5.5 em} \ldots, \frac{Q\zeta_{n-1}}{||Q\eta||} + \frac{\langle Q\eta, Q\zeta_{n-1}\rangle}{||Q\eta||^3}  Q\eta\Bigg) \\
			&= \dV_{g_0}\left(\frac{Q\eta}{||Q\eta||}, \frac{Q\zeta_1}{||Q\eta||}, \ldots, \frac{Q\zeta_{n-1}}{||Q\eta||}\right) \\
			&= \frac{\det Q}{||Q\eta||^n}\dS_{g_0}
		\end{aligned}
	\end{equation*}
	where we used the fact that $\zeta_j \perp \eta$. We can rewrite this as 
	\begin{equation}
		f^*\dS_{g_0} = \frac{dV_{g_0}}{dV_g}||\eta||_g^{-n} \dS_{g_0}.
		\label{InverseThmEq3}
	\end{equation}
	Combining (\ref{InverseThmEq1}),(\ref{InverseThmEq2}) and (\ref{InverseThmEq3}) we have
	\[g(x) = \frac{n}{\Vol(S^{n-1})}\int_{S^*_x M} \frac{dV_{g_0}}{dV_g}||\eta||_g^{-n-2}\ \eta\otimes\eta\ \dS_{g_0}(\eta)\]
	as desired.
	\end{proof}
\end{customthm}
\begin{cor}\label{DensityCor}
	The Bergman metrics are dense in the space of all Riemannian metrics in the $C^0$ topology.
	\begin{proof}
	The theorem shows that
	\[\Met(M) \subset \overline{\bigcup_{N\in\N} \B_N}. \]
	Moreover since by (\ref{BNBNPlusOneEq}) $\B_N \subset \overline{\B_{N+1}}$  we have
	\[\Met(M) \subset \lim_{N\to\infty}\overline{\B_N}.\]
	\end{proof}
\end{cor}

But how do we actually use Theorem \ref{InverseThm} to approximate a given Riemannian metric $g$? We need to quantize the symbol in equation \ref{InverseSymbolEq} to a symmetric and positive-definite order zero pseudodifferential operator.  For this we define
\begin{equation}
\Hilb(g) = \frac{n(n+2)(2\pi)^{n}}{\Vol(S^{n-1})}\frac{dV_{g_0}}{dV_g}B_g^*B_g
\end{equation}
where $B_g$ is the order zero pseudodifferential operator defined by
\[B_g = (\D_{g_0}^{-1}\D_{g})^{-\frac{n+2}{4}}.\]

Then in analogy with the K\"ahler case we we define approximation maps
\[\Hilb_N: \Met(M) \to \I(\H_{\leq N}^*)\]
by
\begin{equation}
	\Hilb_N(g) = \langle \Hilb(g) \cdot,\cdot\rangle
	\label{HilbDef}
\end{equation}
where restriction to $\I(\H_{\leq N}^*)$ is implicit.  Then Theorems \ref{LimitThm} and \ref{InverseThm} immediately imply  
\begin{cor}\label{HilbCor}
	\[E_N\circ\Hilb_N(g) = g + o(1).\]
\end{cor}

\section{Induced Metric on The Space of Metrics}\label{MetricSection}

Using these approximation maps we can naturally induce a Riemannian metric on $\Met(M)$.  Let $\dot g$ be a tangent vector at $g \in \Met(M)$, that is $\dot g$ is a symmetric bilinear form on $T^*M$.  Then we define
\begin{equation}
	||\dot g||^2 = \lim_{N\to\infty} \mu_N^{-n}||D\Hilb_N(\dot g)||^2
	\label{MetMetricEq}
\end{equation}
where the Riemannian metric on the right is the symmetric space metric on $\I(\H_{\leq N}^*)$.  Namely if $\Inner[\dot R]$ is a tangent vector (that is symmetric bilinear form) at $\Inner[R]$ its norm is given by
\[ ||\dot R||^2  = \Tr R^{-1}\dot R R^{-1}\dot R. \]

\begin{customthm}{\ref{MetThm}}
	The induced metric on $\Met(M)$ is positive definite and given by
	\[||\dot g||^2 = \frac{1}{4n(2\pi)^{n}}\int_{S^*M} \left(\Tr_g {\dot g}+ (n+2)\frac{\langle g^{-1}\dot g g^{-1}\xi, \xi\rangle}{||\xi||^2_g}\right)^2 dS_{g_0}(x,\xi).\]
	 Moreover the metric is independent of the choice of quantization in the definition of $\Hilb$.
	\end{customthm}
\begin{proof}
	By definition
	\begin{align*}
		||D\Hilb_N(\dot g)||^2 = \Tr [&(\ToeplitzN{\Hilb(g)})^{-1} \ToeplitzN{D\Hilb(\dot g)} \\
		\cdot&(\ToeplitzN{\Hilb(g)})^{-1} \ToeplitzN{D\Hilb(\dot g)}]
	\end{align*}

	Evaluating this as $N\to\infty$ is an application of the Sz\"ego limit Theorem \cite{Guillemin} which states that 
	\[\Tr [\Toeplitz{\lambda}{B_1} \cdots \Toeplitz{\lambda}{B_k}] = \frac{\lambda^n}{n(2\pi)^n}\int_{S^*M} b_1(x,\xi)\cdots b_k(x,\xi) dS_{g_0}(x,\xi) + o(\lambda^{n}).\]

	We only need to show that if $B_j$ is elliptic then we can put $(\Toeplitz{\lambda}{B_j})^{-1}$ on the left and $b_j^{-1}$ on the right.  We have by Lemma \ref{PiLem}(i)
	\begin{align*}
		|\PiLeq\lambda - \Toeplitz{\lambda}{B}B^{-1}\PiLeq\lambda |_1 &= |\PiLeq\lambda B (I - \PiLeq\lambda) B^{-1} \PiLeq\lambda\ |_1\\
		&\leq |\PiLeq\lambda B (I-\PiLeq\lambda)|_2 \cdot |B^{-1}\PiLeq\lambda|_2 \\
		&= o(\lambda^n).
	\end{align*}
	And so 
	\begin{align*}
	&|\Toeplitz{\lambda}{B^{-1}} - (\Toeplitz{\lambda}{B})^{-1}|_1 \\
	\leq\ &||B^{-1}||\cdot|\Toeplitz{\lambda}{B}(\Toeplitz{\lambda}{B^{-1}} - (\Toeplitz{\lambda}{B})^{-1})|_1 \\
		=\ &o(\lambda)^n.
	\end{align*}
	Thus
	\begin{align*}
		\Tr [(\Toeplitz{\lambda}{B_1^{-1}} &- (\Toeplitz{\lambda}{B_1})^{-1})(\Toeplitz{\lambda}{B_2}\cdots \Toeplitz{\lambda}{B_k})] \\
		&\leq ||B_2\cdots B_k||\cdot |(\Toeplitz{\lambda}{B_1^{-1}} - (\Toeplitz{\lambda}{B_1})^{-1})|_1 = o(\lambda^n).
	\end{align*}
And so by the Szeg\"o theorem
	\[\Tr (\Toeplitz{\lambda}{B_1})^{-1} \cdots \Toeplitz{\lambda}{B_k} = \frac{\lambda^n}{n(2\pi)^n}\int_{S^*M} b_1^{-1}(x,\xi)\cdots b_k(x,\xi) dS_{g_0}(x,\xi) + o(\lambda^{n}).\]
By repeating this argument we have
	\[||D_g\Hilb_N(\dot g)||^2 = \frac{\mu_N^n}{n(2\pi)^n}\int_{S^*M}(\sigma\circ\Hilb(g)(x,\xi))^{-2} (\sigma\circ D\Hilb(\dot g)(x,\xi))^2 dS_{g_0}(x,\xi) + o(\mu_N^{n}).\]

	It remains to compute $\sigma\circ D\Hilb = D(\sigma\circ\Hilb)$ by linearity of the principal symbol map.  Here it is clear that $||\dot g||^2$ depends only on $\sigma\circ\Hilb$ and is, as claimed, independent of the choice of quantization in defining $\Hilb$.
	
	Note that if $\dot g$ is a tangent vector at $g \in \Met(M)$, that is a symmetric covariant 2-tensor, then as $\e \to 0$ we have
	\[||\xi||^2_{g+\e\dot g} = ||\xi||^2_g - \e\langle g^{-1}\dot g g^{-1} \xi, \xi\rangle + O(\e^2).\]

	Thus modulo $\e^2$ we have
	\begin{align*}
		\sigma\circ\Hilb_N (g+\e\dot g) &=(n+2)(2\pi)^{n}(\det g_0^{-1}(g+\e\dot g))^{-1/2} ||\xi||^{-n-2}_{g+\e\dot g} \\
		&= \sigma\circ\Hilb_N (g) (1+\frac{\e}{2}\Tr g^{-1}\dot g)\left(1+\e\frac{n+2}{2}\frac{\langle g^{-1}\dot g g^{-1} \xi, \xi\rangle}{||\xi||^2_{g}}\right).
	\end{align*}

	Thus
	\begin{equation}
		D_{g} (\sigma\circ\Hilb_N) (\dot g) = \frac{\sigma\circ\Hilb_N(g)}{2}\left(\Tr g^{-1}\dot g + (n+2)\frac{\langle g^{-1}\dot g g^{-1} \xi, \xi\rangle}{||\xi||^2_{g}}\right).
		\label{}
	\end{equation}

	We substitute $\tilde g = g^{-1}\dot g g^{-1}$ and note that $\tilde g (x) = 0$ if and only if $\dot g(x) = 0$ and we have
	\begin{equation}\label{GTildeEq}
		||\dot g||^2 = \frac{1}{4n(2\pi)^n}\int_{S^*M} \left(\Tr{g\tilde g} + (n+2)\frac{\langle \tilde g \xi, \xi\rangle}{||\xi||^2_g}\right)^2 dS_{g_0}(x, \xi).\end{equation}
	Since
	\begin{align*}
		|\Tr g\tilde g(x)| &\leq n\sup_{\xi \in ^{g}S^*_x(M)}|\langle \tilde g\xi, \xi\rangle| = n\sup_{\xi \in S^*_x(M)}\frac{|\langle \tilde g\xi, \xi\rangle|}{||\xi||^2_g}
	\end{align*}
	we have the weaker inequality
	\[|\Tr g\tilde g(x)| \leq \sup_{\xi \in S^*_x M}\left|(n+2)\frac{\langle \tilde g \xi, \xi\rangle}{||\xi||^2_g}\right|\]
	with equality if and only if $\tilde g(x) \equiv 0$. Hence the integrand in (\ref{GTildeEq}) is identically zero if and only if $\tilde g \equiv 0$. That is, the metric on $\Met(M)$ is positive definite.
	\end{proof}

\section{The Round Sphere}\label{SphereSection}
Recall that the eigenvalues of the $n$-dimensional round sphere $S^n$ are
\[\mu_N^2 = N(N+n-1)\]
and so $\mu_N \sim N$ as $N\to\infty$.
\begin{customthm}{\ref{SphereThm}}
	Let $(M,g_0)$ be the round sphere. Given an order zero pseudodifferential operator $B$ and an integer $k$ we have as $N\to\infty$ 
	\begin{equation}\label{SphereThmEq}
		\DD\Pi_{N+k} B\Pi_N (x) = (2\pi)^{-n} \left(\int_{-\pi}^\pi e^{-itk} \int_{S^*_x S^n} b\circ G^t(x,\xi) \xi\otimes\xi dS_{g_0}(\xi) dt\right)N^{n+1} + O(N^{n}).
\end{equation}
	\end{customthm}

The proof of Theorem \ref{SphereThm} rests on the fact (\cite{ZelditchZoll}) that the spectral projections kernel on the round sphere is an oscillatory integral (or semiclassical Lagrangian distribution). We recall first the general definition and properties of such objects (following Duistermaat \cite{Duistermaat}) and the calculus in our setting (following \cite{ZelditchZoll}), from which the proof of the theorem will follow.

Let $X$ be a manifold and $i: \Lambda \into T^*X$ be an immersed Lagrangian submanifold of the cotangent bundle of $X$.
Let $\phi:X\times\R^k \to \R$ be a smooth function and 
\begin{equation*}
	C_\phi = \{(x,\theta) \in X \times \R^k \textrm{ ; } d_\theta\phi(x,\theta) = 0\}
\end{equation*}
be the set where $\phi$ is stationary in $\theta$.  Then we say that $\phi$ locally parameterizes $\Lambda$ if the map
\begin{align*}
	i_\phi: C_\phi &\to T^*X	\\
	(x,\theta) &\mapsto (x,d_x\phi)
\end{align*}
has image contained in $i(\Lambda)$.

Then an oscillatory integral with parameter $\tau$ and order $\mu$ associated to $i(\Lambda)$ is a locally finite sum of integrals of the form

\begin{equation*}
	I(x,\tau) = \left( \frac{\tau}{2\pi} \right)^{k/2} \int_{\R^k} e^{i\tau\phi(x,\theta)} a(x,\theta,\tau)d\theta
\end{equation*}
where the amplitude has asymptotic expansion
\begin{equation*}
	a(x,\theta,\tau) \sim \sum_{r=0}^{\infty}a_r(x)\tau^{\mu-r}.
\end{equation*}

The principal symbol of $u$ is a section of $\Omega_{\frac{1}{2}}\otimes \mathcal{M}$, the bundle of half densities on $\Lambda$ tensor the Maslov line bundle. It is defined by

\[\sigma(u)(x_0,\xi_0) = a_0(x_0,\xi_0) \sqrt{d_{C_\phi}}\]
where $d_{C_{\phi}}$ is a certain density supported on $\Lambda$ and depending on a choice of density on $T_{x_0}M$.

If $B$ is a pseudodifferential operator then we have
\begin{equation*}
	\sigma(Bu) = b|_{i(\Lambda)}\sigma(u) + O(\tau^{-1})
\end{equation*}
where $b$ is the principal symbol of $B$.

Consider the embedding 
\[i:S^1\times S^*S^n \into S^*S^n\times S^*S^n\]
given by
\[i(t,x,\xi) = ((x,\xi),G^t(x,\xi))\]
and let $\Gamma$ denote the image which is a Lagrangian submanifold.

Then $\Pi_N(x,y)$ is a semiclassical Lagrangian distribution corresponding to $\Gamma$ with principal symbol given by the half-density
\[i^*\sigma(\Pi_N) = (2\pi)^{-1} N^{\frac{n-1}{2}}e^{iNt}|dt|^\frac{1}{2}|dS_{g_0}|^\frac{1}{2}.\]

When restricted to $S^1\times S^*S^n$ the composition formula for principal symbols of semiclassical Lagrangian distributions associated to $\Gamma$ is (\cite{ZelditchZoll}, p. 430)
\[a*b(t,x,\xi) = \int_{S^1} a(s,x,\xi)b(t-s,G^s(x,\xi))ds.\]

In order to implement the operator $\DD$ we also need to differentiate semiclassical Lagrangian distributions.  We have that if $u$ is a semiclassical Lagrangian distribution of order $k$ associated to $\Lambda$ then
\[\sigma(d_x\otimes d_y u)(x,\xi,y,\eta) = \xi\otimes\eta\ \sigma(d_x\otimes d_y u)(x,\xi,y,\eta)\]
with $d_x\otimes d_y u$ also associated to $\Lambda$ but of order $k+2$ (since each derivative brings down a $k$) and taking values in the covariant 2-tensor bundle.

The final piece of the calculus which we will need is the operation of restriction to diagonal in $S^n \times S^n$. Let $i_\D:S^n\to S^n\times S^n$ be the diagonal embedding $x\mapsto (x,x)$.  Note that the pullback of $\Gamma$ under $i_{\D}$ is simply
\begin{align*} 
	i_{\D}^* \Gamma &= \{(x,\xi+\eta) \in T^*S^n \textrm{ ; } (x,\xi,x,\eta) \in \Gamma \} \\
	&= \{(x,0) \in T^*S^n \} = S^n
\end{align*}
where $S^n$ denotes the zero section in $T^*S^n$.

The symbol of the pullback is thus given by 
\[i^*\sigma(u|_{x=y})(x,0) = C \int_{S^1}i^*\sigma(u)(t,x,0)|dt|^{\frac{1}{2}}.\]

\begin{proof}[Proof of Theorem \ref{SphereThm}]
	Using the above calculus we have that $\Pi_{N+k} B \Pi_N$ is a semiclassical Lagrangian distribution associated to $\Gamma$ with symbol
	\[i^*\sigma(\Pi_{N+k} B \Pi_N)(t,x,\xi) = \xi\otimes\xi \int_{S^1}e^{i[(N+k)t+ks]}b\circ G^s(x,\xi) ds |dt|^{\frac{1}{2}}|dS_{g_0}|^{\frac{1}{2}}.\]
	Then $\DD\Pi_N B \Pi_N$ is an oscillatory integral of order $\frac{n}{2}$ associated to the zero section in $T^*S^n$ and with symbol
	\[\sigma(\DD\Pi_{N+k} B \Pi_N)(x,0) = C \int_{S^1}e^{iks}\int_{S^*_x S^n}b\circ G^s(x,\xi) \xi\otimes \xi dS_{g_0}(\xi) ds |dV_{g_0}|^{\frac{1}{2}}.\]
	Since $i_\D^*\Gamma$ is the zero section and the volume form on the round sphere is uniform we have
	\begin{equation}
		\DD\Pi_{N+k} B \Pi_N(x) = C \int_{S^1}e^{iks}\int_{S^*_xS^n}b\circ G^s(x,\xi) \xi\otimes \xi dS_{g_0}(\xi) ds
	\label{SphereThmEq1}
\end{equation}
with a new constant which we claim to be $(2\pi)^{-n}$. The constant is independent of $B$ and $k$ so we use the fact that when $B=I$ and $k=0$
\[\DD\Pi_{N+k} B \Pi_N (x) = \DD \Pi_N(x) = \frac{N^n\Vol(S^{n-1})}{n(2\pi)^n}g_0(x)\]
where we used Takahashi's theorem \cite{Takahashi} and Weyl's law. On the other hand, in this case by Lemma \ref{TensorIntegralLem} the right hand side of equation \ref{SphereThmEq1} is 
\begin{equation*}
	CN^{n+1} \int_{S^1}\int_{S^*_xS^n}\xi\otimes \xi dS_{g_0}(\xi) ds = CN^{n}\frac{\Vol(S^{n-1})}{n}g_0
	\label{}
\end{equation*}
so we have $C = (2\pi)^{-n}$ as claimed.
\end{proof}

Note that the metrics in equation \ref{SphereThmEq} are indeed symmetric with respect to $x\mapsto -x$ as expected from the parity of spherical harmonics.  When $k=0$ we have the following corollary:

\begin{cor}
	Let $B$ be an order zero pseudodifferential operator which is positive definite and symmetric on $\H_{N}$ for $N$ sufficiently large.  Then we have
	\[\iota_N^*\Inner[B](x) = (2\pi)^{-n}\left(\int_{S^*_xS^n} \xi\otimes\xi \int_{S^1} b\circ G^t(x,\xi)\ dt\ dS_{g_0}(\xi)\right)N^{n+1} + O(N^{n})\]
	as $N\to\infty$.
\end{cor}

Next we show how this theorem can be used to derive Bergman metric asymptotics in the sphere case with a better remainder term.

\begin{theorem}
	Let $(M,g_0)$ be the round sphere and $B$ an order zero pseudodifferential operator which is symmetric and positive definite. Then 	\[ \iota_{\leq N}^*\Inner[B](x) = \left(\frac{1}{(2\pi)^{n}(n+2)}\int_{S^*_x S^n}b(x,\xi) \xi\otimes\xi\ \dS_g(\eta)
	\right)N^{n+2} + O(N^{n+1})\]
as $N\to\infty$.
	\begin{proof}
		Let 
		\[\a_{jk} = \DD \Pi_j B \Pi_k (x) \]
		so that
		\[ \DD \Pi_{\leq N} B \Pi_{\leq N} (x) = \sum_{j,k = 0}^N \a_{jk}.\]
		Now we rearrange the sum by defining
		\[\b_{lm} =  \begin{cases} 
			\a_{m,l+m} & l \geq 0 \\
			\a_{-l+m,m} & l < 0
		\end{cases} \]
		so that $\b_{lm}$ is the $l$-th diagonal above (or below when $l < 0$) the main diagonal of $(a_{jk})$.
		Then by Theorem \ref{SphereThm}, for fixed $l$ and as $m\to\infty$ we have
		\[ \b_{lm} =  \frac{m^{n+1}}{(2\pi)^{n}}\int_{S^1} e^{itl}\left(\int_{S^*_x S^n}(b\circ G^t(x,\eta) \xi\otimes\xi \dS_g(\xi) + O(m^{-1})\right)dt.\]
		Now we can rewrite the sum
		\[ \sum_{i,j=0}^N \a_{ij} = \sum_{m=0}^N \sum_{l=-(N-m)}^{N-m} \b_{lm}.\]
		And by convergence of Fourier series for smooth functions on the circle we have for any positive integer $r$
		\[ \left| \sum_{l=-(N-m)}^{N-m} \b_{lm} - \frac{1}{(2\pi)^{n}}\left(\int_{S^*_x S^n} b(x,\xi) \xi\otimes\xi \dS_g(\xi)\right) m^{n+1} \right| \leq C_r\frac{m^{n+1}}{(1+(N-m))^{-r}} + m^n.\]
		And so
		\begin{align*}
			\DD \Pi_{\leq N} B \Pi_{\leq N} (x) &= \sum_{m=0}^N \sum_{l=-(N-m)}^{N-m} \b_{lm} \\
			&= \frac{N^{n+2}}{(n+2)(2\pi)^{n}}\int_{S^*_x S^n} b(x,\xi) \xi\otimes\xi\ \dS_{g_0}(\xi) + O(N^{n+1})
		\end{align*}
		as claimed.
	\end{proof}
\end{theorem}

\bibliographystyle{amsalpha}
\bibliography{paper}

\end{document}